\documentclass[11pt]{article}
\usepackage{graphicx,hyperref,color}
\usepackage{amsmath,amssymb,amsthm}
\usepackage{enumerate}
\usepackage{fullpage}

\newtheorem{theorem}{Theorem}
\newtheorem{corollary}[theorem]{Corollary}
\newtheorem{lemma}[theorem]{Lemma}
\newtheorem{claim}[theorem]{Claim}
\newtheorem{remark}[theorem]{Remark}

\newtheorem{observation}[theorem]{Observation}
\newtheorem{conjecture}[theorem]{Conjecture}

\date{}

\title{Large induced acyclic and outerplanar subgraphs of 2-outerplanar graph \thanks{This material is based upon work supported by the National Science Foundation under Grant Nos.\ CCF-1252833 and DMS-1359173.  Melissa Sherman-Bennett performed this work while participating in the Summer 2015 REU program in Mathematics and Computer Science at Oregon State University.}}
\author{Glencora Borradaile \\Oregon State University \\ \texttt{glencora@engr.oregonstate.edu}  \and Hung Le \\Oregon State University \\ \texttt{lehu@onid.oregonstate.edu}  \and Melissa Sherman-Bennett \\University of California, Berkeley \\ \texttt{m\_shermanbennett@berkeley.edu}}
\begin{document}
\maketitle
\begin{abstract}
Albertson and Berman conjectured that every planar graph has an induced forest on half of its vertices.  The best known lower bound, due to Borodin, is that every planar graph has an induced forest on two fifths of its vertices.  In a related result, Chartran and Kronk, proved that the vertices of every planar graph can be partitioned into three sets, each of which induce a forest.

We show tighter results for 2-outerplanar graphs. We show that every 2-outerplanar graph has an induced forest on at least half the vertices by showing that its vertices can be partitioned into two sets, each of which induces a forest. We also show that every 2-outerplanar graph has an induced outerplanar graph on at least two-thirds of its vertices, provided that the connected components of the inner layer are two-connected.  
\end{abstract}

\section{Introduction}
For many optimization problems, finding subgraphs with certain properties is a key to developing algorithms with efficient running times or bounded approximation ratios.  For example, balanced separator subgraphs support the design of divide-and-conquer algorithms for minor-closed graph families~\cite{LT79,LT80,GHT84,KR10} and large subgraphs of low-treewidth\footnote{Formal definitions of graph theoretic terms will be given at the end of this section.} support the design of approximation schemes, also for minor-closed graph families~\cite{Baker94,DHM07,Eppstein00}.  In the area of graph drawing, one often starts by drawing a subgraph that is somehow easier to draw than the entire graph (such as a planar graph or a tree) and then adds in remaining graph features~\cite{BatEadTam-98}; the larger the subgraph, the bigger the head-start for drawing and the more structure the subgraph has, the easier the subgraph will be to draw.  

In this paper we are concerned with finding large induced subgraphs, in particular large induced forests and large induced outerplanar subgraphs, of input planar graphs.  We are motivated both by the intrigue of various conjectures in graph theory but also by the impact that graph theoretic results have on the design of efficient and accurate algorithms.  In particular, many algorithms that are specifically designed for planar graphs rely on deep graph theoretic properties of planar graphs~\cite{Whitney32,Whitney33}.

\subsection{Large induced forests of planar graphs: known results}
 Albertson and Berman conjectured that every planar graph has an induced forest on at least half of its vertices~\cite{AB}; $K_4$ illustrates that this would be the best possible lower bound.  A proof of the Albertson-Berman Conjecture would, among other things, provide an alternative proof, avoiding the 4-Color Theorem, that every planar graph has an independent set with at least one-quarter of the vertices.

The best-known lower bound toward the Albertson-Berman Conjecture has stood for 40 years: Borodin showed that planar graphs are {\em acyclically 5-colorable} (i.e. have a 5-coloring, every two classes of which induce a forest), thus showing that every planar graph has an induced forest on at least two-fifths of its vertices~\cite{Borodin79}.  This is the best lower bound achievable toward the Albertson-Berman Conjecture via acyclic colorings as there are planar graphs which do not have an acyclic 4-coloring (for example $K_{2,2,2}$ or the octahedron). 

The Albertson-Berman Conjecture has been proven for certain subclasses of planar graphs. Shi and Xu~\cite{SX16} showed that the  Albertson-Berman Conjecture holds when $m < \lfloor 7n/4 \rfloor$ where $m$ and $n$ are the number of edges and vertices of the graphs, respectively. Hosono showed that outerplanar graphs have induced forests on at least two-thirds of the vertices~\cite{Hosono} and Salavatipour showed that every triangle-free planar graph on $n$ vertices has an induced forest with at least $\frac{17n+24}{32}$ vertices~\cite{Salavatipour}, later improved to $\frac{6n+7}{11}$ by Dross, Montassier and Pinou~\cite{DMP14} and to $\frac{5n}{9}$ by Le~\cite{Le16}. In bipartite planar graphs, the best bound on the size of the largest induced forest is $\frac{4n}{7}$ by Wang, Xie and Yu~\cite{WXY16}. 

One direction toward proving the Albertson-Berman Conjecture is to partition the vertices of graph $G$ into sets such that each set induces a forest; the minimum number, $a(G)$, of such sets is the {\em vertex arboricity} of $G$. This implies that $G$ has an induced forest with at least $1/a(G)$ of its vertices.
Chartran and Kronk first proved that all planar graphs have vertex arboricity at most 3~\cite{CK}. Raspaud and Wang proved that $a(G) \leq 2$ if $G$ is planar and either $G$ has no $4$-cycles, any two triangles of $G$ are at distance at least $3$, or $G$ has at most 20 vertices; they also illustrated a 3-outerplanar graph on 21 vertices with vertex arboricity 3~\cite{RW}. Yang and Yuan~\cite{YY07} proved that $a(G) \leq 2$ if $G$ is planar and has diameter at most $2$.

\subsection{Outline of our results}
In this paper, we show that 2-outerplanar graphs have vertex arboricity at most 2, thus showing that they satisfy the Albertson-Berman Conjecture and closing the gap for planar graphs with vertex arboricity 2 versus 3 left by Raspaud and Wang's work (Section~\ref{sec:2-2}).  We also show that every 2-outerplanar graph has an induced outerplanar graph on at least two-thirds of its vertices, assuming the connected components of the inner layer are two-connected and propose a few related conjectures (Section~\ref{sec:2-to-1}).  

\subsection{Definitions}

We use standard graph theoretic notation~\cite{Diestel05}.  In this paper, all graphs are assumed to be finite and simple (without loops or parallel edges).  $G[S]$ denotes the \emph{induced subgraph} of graph $G$ on vertex subset $S$: the graph having $S$ as its vertices and having as edges every edge in $G$ that has both endpoints in $S$. Equivalently, $G[S]$ may be constructed from $G$ by deleting every vertex and incident edges that is not in $S$. We use $d_H(v)$ to denote degree of vertex $v$ in graph $H$ and $|H|$ to denote the number of vertices of graph $H$.

\paragraph{Block-Cut Tree.} A \emph{block} of a graph $G$ is a maximal two-connected component of $G$. A \emph{block-cut tree} $\mathcal{T}$ of a connected graph $G$ is a tree where each vertex of $\mathcal{T}$ corresponds to a block and there is an edge between two vertices $X,Y$ of $\mathcal{T}$ if two blocks $X$ and $Y$ share a common vertex or are incident to a common edge. 
\paragraph{Planar graphs.} A graph $G$ is \emph{planar} if it can be drawn (embedded) in the plane without any edge crossings.  Although a planar graph may have many different embeddings, throughout this paper, we will assume that we are given a fixed embedding of the graph.  A {\em face} of a planar graph is connected region of the complement of the image of the drawing.  There is one {\em infinite} face, which we denote by $f_\infty$. We denote the boundary of $f_\infty$, which is the boundary of $G$, by $\partial G$. We say that a vertex $v$ is enclosed by a cycle $C$ if every curve from the image of $v$ to an infinite point must cross the image of $C$.

\paragraph{Planar duality.}

Every planar graph $G$ has a corresponding dual planar graph $G^*$:  the vertices of $G^*$ correspond to the faces of $G$ and the faces of $G^*$ correspond to the vertices of $G$; an edge of $G^*$ connects two vertices of $G^*$ if the corresponding faces of $G$ share an edge (in this way the edges of the two graphs are in bijection). 

\paragraph{Outerplanarity.}

A non-empty planar graph $G$ with a given embedding is \emph{outerplanar} (or $1$-\emph{outerplanar}) if all vertices are in $\partial G$. A planar graph is $k$-\emph{outerplanar} for $k > 1$ if deleting the vertices in $\partial G$ results in a $(k-1)$-outerplanar graph.  A $k$-outerplanar graph has a natural partition of the vertices into $k$ {\em layers}:  $L_1$ is the set of vertices in $\partial G$;  $L_i$ is the set of vertices in the boundary of $G \setminus \cup_{j < i} L_j$. We denote $G(V,E)$ by $G(L_1,\ldots, L_k; E)$ if $G$ is $k$-outerplanar. For a 2-outerplanar graph, we define the \emph{between degree} of a vertex $v \in L_i$ to be the number of adjacent vertices in $L_j, j \not= i$.

\paragraph{Facial Block.}  Let $\mathcal{C}$ be the set of facial cycles bounding finite faces of $G[L_1]$. For each $C \in \mathcal{C}$, let $S_C$ be the set of vertices enclosed by $C$ in $G$. Then we call the graph $G[C \cup S_C]$ a facial block of $G$. 

\section{2-outerplanar graphs have vertex-arboricity 2} \label{sec:2-2}

In this section, we prove:

\begin{theorem}\label{thm:2-arbor} 
If $G$ is a 2-outerplanar graph, then the vertex arboricity of $G$ is at most $2$: $a(G) \leq 2$.
\end{theorem}

We call a set of vertex-disjoint induced forests of $G$ \emph{induced p-forests} if their vertices partition the vertex set of $G$. We consider a counterexample graph $G$ of minimal order. By studying the structure of this minimal counterexample, we will derive a contradiction. Let $e$ be an edge that is not in $G$. We observe:

\begin{observation} \label{obs:disk-tri}
If $a(G\cup \{e\}) \leq 2$, then $a(G) \leq 2$.
\end{observation}

Observation~\ref{obs:disk-tri} allows us to assume w.l.o.g.\ that $G$ is connected (by adding edges between components while maintaining 2-outerplanarity) and that $G$ is a disk triangulation, i.e., that every face except the outer face of $G$ is a triangle (by adding edges inside non-triangular faces while maintaining 2-outerplanarity). Let $L_1, L_2$ be the bipartition of the vertices of $G$ into layers.

\begin{observation} \label{obs:L1-2-connected}
$G[L_1]$ is two-connected.
\end{observation}
\begin{proof}
Suppose otherwise. Let $v$ be a cut vertex of $G[L_1]$. Then $v$ is also a cut vertex of $G$ since $L_1$ is the outermost layer. Let $B_1$, $B_2$ be two induced subgraphs of $G$ that share the cut vertex $v$ and $V(B_1) \cup V(B_2) = V$. Since $G$ is minimal, we can partition each $B_i$ into two induced forests $F_{1i}$ and $F_{2i}$, $1 \leq i \leq 2$.~W.l.o.g, we assume that $V(F_{11})\cap V(F_{12}) = \{v\}$. Then, $F_{11} \cup F_{12}$ and $F_{21} \cup F_{22}$ are two induced p-forests of $G$, contradicting that $G$ is a counter-example.      
\end{proof}

\begin{claim} \label{clm:no-deg-3}
Every vertex in $G$ has degree at least 4.
\end{claim}
\begin{proof}
Suppose $G$ has a vertex $v$ of degree at most 3.  Since $G$ is a minimal order counterexample and $G - v$ is a 2-outerplanar graph,  $a(G- v) \leq 2$.  Let $F_0$ and $F_1$ be two induced p-forests of $G- v$.  Since $v$ has at most 3 neighbors in $G$, one of $F_0$ or $F_1$, w.l.o.g.\ say $F_0$, contains at most one of these neighbors.  Therefore $F_0 \cup \{v\}$ is a forest of $G$ and $F_0 \cup \{v\}, F_1$ are two induced p-forests of $G$, contradicting that $G$ is a counterexample.
\end{proof}

By Observation~\ref{obs:L1-2-connected}, $\partial G[L_1]$ is a simple cycle. Thus, the graph, say $H_1^*$, of $G[L_1]$ obtained from the dual graph of $G[L_1]$ by removing the dual vertex corresponding to the infinite face of $G[L_1]$ is a tree.  Let $B$ be a facial block of $G$ that has the boundary cycle corresponding to a leaf of $H_1^*$. Then, either $\partial B$ has exactly one edge not in $\partial G$ or $B \equiv G$. In the former case, let $e_B$ be the shared edge; in the later case, let $e_B$ be any edge of $B$. Denote $L_2^B = L_2\cap V(B)$. We have: 

\begin{claim}\label{clm:l2b-non-empty}
$|L_2^B| \geq 2$.
\end{claim}  
\begin{proof}
If $|L_2^B| = 0$, then $B$ is a triangle since $G$ is a disk-triangulation and vertices have degree at least $4$. Then, the vertex of $B$ that is not an endpoint of $e_B$ has degree 2 in $G$, contradicting Claim~\ref{clm:no-deg-3}. If $L_2^B = \{v\}$,  by Claim~\ref{clm:no-deg-3}, $v$ has at least four neighbors in $L_1$ and thus, at least one neighbor $u$ of $v$ in $L_1$ is not an endpoint of $e_B$. Then the degree of $u$ in $G$ is 3, contradicting Claim~\ref{clm:no-deg-3}.
\end{proof}

\begin{claim} \label{clm:non-cut-deg}
Let $v$ be a vertex in $L_2^B$ that has between degree at least $3$. Then, either $v$ is a cut vertex of $G[L^B_2]$ or $v$ is adjacent to both endpoints of $e_B$.
\end{claim}
\begin{proof}
Let $v_1,v_2,v_3$ be neighbors of $v$  in $\partial B$ in clockwise  order around $v$. Let $\partial B[v_i,v_j]$ be the clockwise segment of $\partial B$ from $v_i$ to $v_j$, $i\not= j$. We define $C_{ij} =  \partial B[v_i,v_j] \cup \{vv_i,vv_j\}$, which is a cycle of $B$. Assume $v$ is not a cut vertex, at most one cycle of $\{C_{12},C_{23}, C_{31}\}$ encloses a vertex of $L_2^B$, say $C_{31}$. Thus, $v_2$ is only adjacent to $v$ and two other neighbors, say $v_1',v_3'$, of $\partial B$. Since $C_{12}$ and $C_{23}$ enclose no vertex of $L_2^B$, $vv'_1v_2$ and $vv_2v_3'$ are faces of $G$. If neither $v'_1v_2 = e_B$ nor $v_2v'_3 = e_B$, then $d_G(v_2) = 3$,  contradicting Claim~\ref{clm:no-deg-3}.
\end{proof}
Suppose $v \in L^B_2$ is such that $d_{G[L_2^B]}(v) = 1$. By Claim~\ref{clm:no-deg-3}, $d_G(v) \geq 4$ so  $v$ has between degree at least $3$. Thus, by Claim~\ref{clm:non-cut-deg}, we have:
\begin{observation}\label{obs:in-degree-one}

If there exists $v \in L_2^B$ such that $d_{G[L^B_2]}(v) = 1$, then $v$ must be adjacent to both endpoints of $e_B$. 
\end{observation} 

Let $x_B,y_B$ be the endpoints of $e_B$.  Since $G$ is a triangulation, there is a vertex $v \in L_2^B$ such that $vx_By_B$ is a face of $G$. We call $v$ the \emph{separating vertex} of $B$.

\begin{claim} \label{clm:e-B-cut}
If $v' \not= v$ is a vertex in $L_2^B$ that is adjacent to both endpoints of $e_B$, then, $v'$ is a cut vertex of $L_2^B$.
\end{claim}

\begin{proof}
We will prove that $v'$ has at least one neighbor in $L_2^B$ inside the triangle $v'x_By_B$  and at least one neighbor in $L_2^B$ outside the triangle $v'x_By_B$; thus $v'$ is a cut vertex of $L_2^B$. 

By planarity, the triangle $v'x_By_B$ encloses $v$. Let $C_{vv'} = \{v,x_B,v',y_B\}$ which is a cycle of $G$.  Since $G$ is a disk triangulation and the edge $x_B,y_B$ is embedded outside $C_{vv'}$, there must be an edge or a path inside $C_{vv'}$ connecting $v$ and $v'$. Thus, $v'$ has at least one neighbor in $L^B_2$ inside the triangle $v'x_By_B$.

Suppose that the cycle $C_{v'} = \{\partial B \setminus e_B\} \cup \{v'x_B,v'y_B\}$ does not enclose any vertex of $L^B_2$. Since $B$ is a facial block that only has $e_B$ as a possible edge not in $\partial G$, every vertex in $C_{v'} \setminus \{v'\}$ must have $v'$ as a neighbor and has degree 3, contradicting Claim~\ref{clm:no-deg-3}. Thus, $C_{v'}$ must enclose at least one vertex of $L^B_2$. That implies $v'$ has at least one neighbor in $L^2_B$ outside the triangle $v'x_By_B$ as desired. 
\end{proof}

Since every cut vertex of $L_2^B$ has degree at least $2$ in $G[L_2^B]$, by Claim~\ref{clm:e-B-cut} and Observation~\ref{obs:in-degree-one}, we have:

\begin{observation} \label{obs:deg-one-unique}
Only the separating vertex $v$ of $B$ can have $d_{G[L_2]}(v) = 1$.
\end{observation}

If the block-cut tree of $G[L_2^B]$ has at least two vertices, let $K$ be a leaf block of $G[L_2^B]$ that does not contain the separating vertex of $B$. In this case, by Observation~\ref{obs:deg-one-unique}, $|K| \geq 3$.  Otherwise, let $K  = G[L_2^B]$. We refer to the cut vertex of $K$ in the former case and the separating vertex of $B$ in the latter case as the \emph{separating vertex} of $K$. By Claim~\ref{clm:non-cut-deg}, we have:

\begin{observation} \label{obs:K-non-cut-deg}
Non-separating vertices of $K$ have between degree at most $2$. 
\end{observation}

We call a triangle $abc$ of $K$ a \emph{critical triangle} with top $c$ if $d_K(c) = 2$ and $c$ is non-separating. By Observation~\ref{obs:K-non-cut-deg} and Claim~\ref{clm:no-deg-3}, $c$ has exactly two neighbors in $L_1$, that we denote by $d,e$ (see Figure~\ref{fig:critical-triangle}). Since $G$ is a disk triangulation, two edges $da$ and $eb$ are edges of $G$. 
\begin{figure}[tbh]
  \centering
    \includegraphics[height=1.3in]{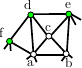}
      \caption{The critical triangle $abc$ and two neighbors $d,e$ of $c$ in $L_1$. Hollow vertices are in $L_2$. }
       \label{fig:critical-triangle}
\end{figure}

\begin{claim} \label{clm:at-least-deg-5}
Vertices $d$ and $e$ have degree at least 5.
\end{claim}
\begin{proof}
Neither $d$ nor $e$ has degree less than 4 by Claim~\ref{clm:no-deg-3}. For contradiction, w.l.o.g, we assume that $d_G(d) = 4$.  Recall $a,c, e$ are three neighbors of $d$. Let $f$ be the only other neighbor of $d$. Since $a,c \in L_2$ and $\partial G$ is a simple cycle (Observation~\ref{obs:L1-2-connected}), $f$ must be in $L_1$ (see Figure~\ref{fig:critical-triangle}). Since $G$ is a disk triangulation, $af \in E(G)$. Let $G'$ be the graph obtained from $G$ by contracting $fd$ and $dc$ and removing parallel edges. Then $G'$ is a minor of $G$ (and so is 2-outerplanar) with fewer vertices. Let $F_0,F_1$ be two induced p-forests of $G'$ that exist by the minimality of $G$. Without loss of generality, we assume that $f \in F_0$. We have two cases:
\begin{enumerate}
\item If $b \in F_0$, then $a,e \in F_1$. If $bf \not\in G$, adding $c,d$ to $F_0$ does not destroy the acyclicity of $F_0$ in $G$. Thus, $F_0 \cup \{c,d\}, F_1$ are two induced $p$-forests of $G$. If $bf \in G$, the cycle $\{b,f,d,c\}$ separates $a$ from $e$ so $a$ and $e$ are in different trees in $F_1$. Thus, $F_0 \cup \{c\}, F_1 \cup \{d\}$ are two induced p-forests of $G$.   
\item Otherwise, $b \in F_1$. We have three subcases:
	\begin{enumerate}
		\item If $a,e$ are both in $F_0$, then $F_0,F_1 \cup  \{c,d\}$ are two induced p-forests of $G$.
		\item If $a,e \in F_1$, then, $F_0 \cup \{c,d\}, F_1$ are two induced p-forests of $G$.
		\item  If $a,e$ are in different induced p-forests of $G'$, then,  $F_0 \cup \{c\}, F_1 \cup \{d\}$ are two induced p-forests of $G$.
	\end{enumerate}	 
\end{enumerate}
In each case, the resulting p-forests contradict that $G$ is a minimal order counter example. 
\end{proof}

\begin{claim} \label{clm:K-non-triangular}
$|K| \geq 4$.
\end{claim}
\begin{proof}
If $K \not= G[L_2^B]$, as noted in the definition of $K$, $|K| \geq 3$. If $K = G[L_2^B]$, then by Claim~\ref{clm:l2b-non-empty}, $|K| \geq 2$ and by Observation~\ref{obs:deg-one-unique}, $|K| \geq 3$. Suppose that $|K| = 3$. Then, $K$ is a triangle. Let $u,w$ be two neighbors of the separating vertex $v$ in $K$. Then, $wuv$ is a critical triangle with top $u$ (or $w$). By Claim~\ref{clm:no-deg-3} and Observation~\ref{obs:K-non-cut-deg}, $u$ and $w$ both have between degree 2. Thus, $u$ and $w$ have a common neighbor on $L_1$ which therefore has degree $4$, contradicting Claim~\ref{clm:at-least-deg-5}. 
\end{proof}

Suppose that $a$ and $b$ of a critical triangle $abc$ with top $c$ of $K$ have a common neighbor $f$ in $L_2$. We have: 

\begin{claim} \label{clm:b-cut-vertex}
If $fa$ (resp. $fb$) is in $\partial G[L_2^B]$, then $a$ (resp. $b$) must be the separating vertex. 
\end{claim}
\begin{proof}
For a contradiction (and w.l.o.g), we assume that $fa \in \partial G[L_2^B]$ and $a$ is  non-separating. See Figure~\ref{fig:cut-vertex-critical}. Let $G'$ be the graph obtained from $G$ by contracting $ac$ and $ce$ and removing parallel edges. Then, $G'$ is a minor of $G$ (and so is 2-outerplanar) with fewer vertices. Let $F_0,F_1$ be two induced p-forests of $G'$, which are guaranteed to exist by the minimality of $G$.  Without loss of generality, we assume that $f\in F_0$. We consider two cases:
\begin{enumerate}
\item If $e \in F_0$, then $d,b$ are in $F_1$. If edge $fe \not\in G$, then, $F_0 \cup \{a,c\}, F_1$ are two induced p-forests of $G$. If $fe \in G$, cycle $\{f,a,c,e\}$ separates $d$ from $b$ so $d$ and $b$ are in different trees of $F_1$. Thus, $F_1 \cup \{c\}, F_0 \cup \{a\}$ are two induced p-forests of $G$. 
\item Otherwise, $e \in F_1$. We have three subcases:
	\begin{enumerate}
		\item If $b,d$ are both in $F_0$, then $F_0, F_1 \cup \{a,c\}$ are two induced p-forests of $G$.
		\item If $b,d$ are both in $F_1$, then $F_0 \cup \{a,c\}, F_1$ are two induced p-forests of $G$.
		\item If $b,d$ are in different forests of $G'$, then, $F_0 \cup \{c\}, F_1 \cup \{a\}$ are two induced p-forests of $G$.
	\end{enumerate}
\end{enumerate}
In each case, the resulting p-forests contradicts that $G$ is a minimal order counter example. 
\end{proof}

\begin{figure}[tbh]
  \centering
    \includegraphics[height=1.3in]{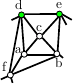}
      \caption{The critical triangle $abc$ with edge $fa \in \partial G[L_2]$. Hollow vertices are in $L_2$. }
       \label{fig:cut-vertex-critical}
\end{figure}

If the edge $fb$ is shared with another critical triangle $fbg$ with top $g$, then we call $\{abc,bfg\}$ a \emph{pair of critical triangles}. See Figure~\ref{fig:critical-pair}. Note that we are assuming that $f$ is a common neighbor of $a$ and $b$ in $L_2$.  

\begin{figure}[tbh]
  \centering
    \includegraphics[height=1.3in]{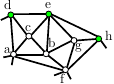}
      \caption{A pair of critical triangles $abc$ and $bfg$. Hollow vertices are in $L_2$}
       \label{fig:critical-pair}
\end{figure}

\begin{claim}\label{clm:critic-pair}
If there exists a pair of critical triangles $abc$ and $bfg$ in $K$, then $b$ must be the separating vertex of $K$.
\end{claim}
\begin{proof}
Note that neither $c$ nor $g$ can be the separating vertex by definition of critical triangles. Suppose for contradiction that $b$ is non-separating. Let $d,e$ be two neighbors of $c$ as defined above and $i$ and $h$ be the neighbors of $g$ in $L_1$. We first argue that $i \equiv e$. Suppose otherwise. Since $G$ is a disk triangulation, $ec,eb, ig, ih, ib$ are edges of $G$. Let $P$ be the subpath of $\partial G$ between $e$ and $i$ that does not contain $d$ and $h$. Note that $P$ could simply be edge $ei$. Since $B$ is a facial block that shares at most one edge with other facial blocks and $b$ is non-separating, $e$ has exactly one neighbor on $P$. That implies $e$ would have degree $4$, contradicting Claim~\ref{clm:at-least-deg-5}.

We also note that $h \not= d$ (for otherwise, $e$ would not be in $L_1$) and $hf \in E(G)$. See Figure~\ref{fig:critical-pair}. Let $G'$ be the graph obtained from $G$ by contracting $ec,eb,eg$ and $eh$ and removing parallel edges. Thus, $G'$ is a minor of $G$ with fewer vertices.  By minimality, $G'$ has two induced p-forests $F_0,F_1$. Without loss of generality, we assume that $a \in F_0$. We will reconstruct two induced p-forests of $G$  by considering two cases:

\begin{enumerate}
\item If $h \in F_0$, then $d,f \in F_1$. If edge $ah \in G$, then, by planarity, $d$ and $f$ are in different trees of $F_1$. Thus, $F_1 \cup \{e,b\}$ has no cycle which implies $F_1\cup \{e,b\}, F_0\cup \{c,g\}$ are two induced p-forests in $G$. Otherwise, $F_0\cup \{b,e\}$ has no cycle. Thus, $F_0\cup \{b,e\}, F_1 \cup \{c,g\}$ are two induced p-forests in $G$. 
\item Otherwise, $h \in F_1$. We have four subcases:
	\begin{enumerate}
		\item If $d,f$ are both in $F_1$, then $F_0 \cup \{c,e,g\}, F_1 \cup \{b\}$ are two induced p-forests of $G$.
		\item If $d,f$ are both in $F_0$, then $F_0 \cup \{e\}, F_1 \cup \{b,c,g\}$ are two induced p-forests of $G$.
		\item If $d \in F_0, f \in F_1$, then $F_0 \cup \{e,g\}, F_1 \cup \{b,c\}$ are two induced p-forests of $G$.
		\item If $d \in F_1, f \in F_0$, then $F_0 \cup \{c,g\}, F_1 \cup \{b,e\}$ are two induced p-forests of $G$.		
	\end{enumerate}
\end{enumerate}
Thus, in all cases, the resulting induced p-forests contradict that $G$ is a counterexample.
\end{proof}

We are now ready to complete the proof of Theorem~\ref{thm:2-arbor} by considering a triangle of $K$ of $G[L_2^B]$, say $uvw$, containing the separating vertex $v$ of $K$ and has the most edges in common with $\partial K$. Since $v$ is separating, $uvw$ contains at least one edge in $\partial K$. We note that $K^*\setminus (\partial K)^*$ where $(\partial K)^*$ is the dual vertex of the infinite face of $K$, is a tree that we denote by $T^*_K$. Recall that $K$ is a block of $G[L^B_2]$. We root $T^*_K$ at the vertex corresponding to the triangle $uvw$. Consider the deepest leaf $x^* \in T^*_K$ and its parent $y^*$. Let $abc$ be the triangle corresponding to $x^*$ such that the dual edge of $ab$ is $x^*y^*$. Then $d_K(c) = 2$. Since $K \geq 4$, $abc \not\equiv uvw$ and thus, it is a critical triangle with top $c$. Let $abf$  be the triangle that corresponds to $y^*$. Note here it may be that $abf \equiv uvw$.  We have three cases:

\begin{enumerate}
\item If $d_{T^*_K}(y^*) = 1$, then $abf \equiv uvw$. Thus, two edges $fa,fb$ are both in $\partial K$ but only one of the two vertices $a,b$ can be the separating vertex of $K$. This contradicts Claim~\ref{clm:b-cut-vertex}.
\item If $d_{T^*_K}(y^*) = 2$, then exactly one of two edges $af, bf \in \partial K$; w.l.o.g, we assume that $bf \in \partial K$. Then, by Claim~\ref{clm:b-cut-vertex}, $b$ must be the separating vertex of $K$. Thus, only two triangles $abc$ and $abf$ contain the separating vertex. Since $uvw$ is the triangle containing the separating vertex with most edges in $\partial K$, $uvw \equiv abc$, contradicting our choice of triangle $abc$.
 \item Otherwise, we have $d_{T^*_K}(y^*) = 3$. Then, none of $\{ab,bf,af\}$ is in $\partial K$, so $abf \not\equiv uvw$. Let $z^*$ and $t^*$ be the other two neighbors of $y^*$ in $T^*_K$ with $t^*$ as the parent of $y^*$. Then, $x^*$ and $z^*$ have the same depth. By our choice of $x^*$, $z^*$ must also be a leaf. Thus, the triangle, say $bfg$, corresponding to $z^*$ is critical. Thus $\{abc, bfg\}$ is a pair of critical triangles. Since $t^*$ is the parent of $y^*$, $b$ cannot be the separating vertex of $K$, contradicting Claim~\ref{clm:critic-pair}.
\end{enumerate}
This completes the proof of Theorem~\ref{thm:2-arbor}. \qed

\section{2-outerplanar graphs have large induced outerplanar graphs}\label{sec:2-to-1}

In this section, we prove:
\begin{theorem} \label{thm:2outer}
Let $G$ be a 2-outerplanar graph on n vertices, and let $L_1, L_2$ be the partition of a 2-outerplanar graph into layers.  Suppose that each connected component of $G[L_2]$ is two-connected. $G$ has an induced outerplanar subgraph on at least $\frac{2n}{3}$ vertices whose outerplanar embedding is induced from $G$.
\end{theorem}

\begin{remark}
	In the first version of this paper, Theorem~\ref{thm:2outer} appeared without the assumption that each connected component of $G[L_2]$ is two-connected. We thank D'Elia and Frati for pointing out the necessity of this assumption to our argument, particularly in Lemma~\ref{lm:match} below. As D'Elia--Frati show in \cite{DF}, this assumption can be removed.
\end{remark}

Note that $\partial G[L_i]$ is a cactus graph (every edge is in at most 1 cycle). As in Section~\ref{sec:2-2}, we assume w.l.o.g\ that $G$ is connected and a disk triangulation. This gives us:

\begin{observation} \label{obs:face-uvw}
If $uv$ is an edge in $\partial G[L_2]$, then there exists $w\in L_1$ such that $uvw$ is a face.
\end{observation}

\begin{observation} \label{obs:between-deg-at-least-1}
The between degree of every vertex in $L_2$ is at least 1. 
\end{observation}

\begin{lemma} \label{lm:twoext}
If $v\in L_2$ has between degree 1, then it is incident to exactly two edges in $\partial G[L_2]$.
\end{lemma}

\begin{proof}
Let $u$ be $v$'s neighbor in $L_1$. Since $G$ is a disk triangulation, there exist two triangular faces, say $xuv$ and $yuv$, containing the edge $uv$. As the between degree of $v$ is 1, $x$ and $y$ are in $L_2$, and the edges $xv$ and $yv$ are in $\partial G[L_2]$. Therefore, $v$ is incident to at least two edges in $\partial G[L_2]$. 

Suppose for the sake of contradiction that $v$ is incident to more than two edges in $\partial G[L_2]$.  Let $w$ be a neighbor of $v$ such that $w \not\in \{x,y\}$. Then by Observation~\ref{obs:face-uvw}, there exists $s\in L_1$ such that $vws$ is a face. Since $w\notin \{x, y\}$ and $G$ is simple, $s \neq u$. This implies $v$ has between degree at least $2$; contradicting that $v$'s between degree is $1$. 
\end{proof}

\begin{lemma} \label{lm:uv-edge-partial-B}
If a facial block $B$ contains a vertex in $L_2$, then endpoints of any edge $uv \in \partial B$ are adjacent to a common vertex in $L^B_2$.
\end{lemma}
\begin{proof}
Since $G$ is a triangulation, there is a vertex $w \in B$ such that $uvw$ is a triangular face; thus $uvw$ contains no vertex of $L_2$. Suppose that $w \in \partial B$, then $uvw$ is an induced cycle of $G[L_1]$. Thus, $uvw \equiv \partial B$; contradicting that $B$ contains a vertex in $L_2$.
\end{proof}

\begin{lemma} \label{lm:match}
 There exists a matching $M\subseteq \partial G[L_2]$ with the following property: 
\begin{equation} \label{eq:cond}
\mbox{If } v\in L_2 \setminus V(M) \mbox{ then }v \mbox{ has between degree at least 2}. 
\end{equation}

\end{lemma}

\begin{proof}
Let $L$ be the set of vertices of between degree 1 in $L_2$. We proceed by strong induction on $|L|$. If $|L|=0$, any matching $M \subseteq \partial G[L_2]$ has property \eqref{eq:cond}.

Let $v\in L$ and $u \in L_2 \setminus L$ be vertices such that $uv \in \partial G[L_2]$; $v$ exists by Lemma~\ref{lm:uv-edge-partial-B}. By Lemma~\ref{lm:twoext}, $v$ has exactly one other neighbor $w$ such that $vw \in \partial G[L_2]$. Contract $uv$ and $wv$ to $v$ and delete parallel edges and loops; let the resulting graph be $G'$. As each connected component of $G[L_2]$ is two-connected, $G'$ is again a disk triangulation. Since $v$ now has between degree at least $2$, the number of vertices of between degree $1$ in $G'$ is strictly less than $|L|$. Therefore, by the inductive hypothesis, there exists a matching $M'\subseteq \partial G'[L_2]$ with property \eqref{eq:cond}. Now, consider $M'$ as a matching in $\partial G[L_2]$. 

If $v$ is not covered by $M'$ in $G'$, then $u, v, w$ are not covered by $M'$ in $G$. Then, $M' \cup \{vw\}$ is a matching and has property \eqref{eq:cond}, since $u$ has between degree at least 2 as argued above. 

If $v x \in M'$ for some $x \in G'$, then either $u x \in \partial G[L_2]$ or $w x \in \partial G[L_2]$. In the first case, let $M =( M' \setminus \{v x\}) \cup \{u x, v w\}$; in the second, let  $M =( M' \setminus \{v x\}) \cup \{w x, u w\}$. In both cases, $M$ is a matching of $\partial G[L_2]$ with property \eqref{eq:cond}. 
\end{proof}

We are now ready to prove Theorem~\ref{thm:2outer}. To find the vertices inducing a large outerplanar graph in $G$, we delete vertices in $L_1$ until all vertices in $L_2$ are ``exposed'' to the external face. To ensure that the resulting outerplanar graph is sufficiently large, we delete vertices in $L_1$ that expose 2 vertices in $L_2$ or otherwise ensure 2 vertices will be included in the outerplanar graph.

Let $M$ be a matching as guaranteed by Lemma~\ref{lm:match}. We create a list $K$ of triples such that each vertex in $L_2$ occurs in exactly one triple. For each $u\in L_2$ not covered by $M$, $u$ has between degree at least $2$, and we add $\{u, v, w\}$ to $K$, where $v, w$ are neighbors of $u$ in $L_1$. For each edge $xy\in M$, by Observation~\ref{obs:face-uvw}, there exists $z\in L_1$ such that $xyz$ is a face, and we add $\{x, y, z\}$ to $K$.

We then delete vertices from $L_1$ as follows:
\begin{tabbing} 
\hspace{16pt} \= 1. While \= there exists $\{u, v, w\}\in K$ such that $\{u, v, w\}\cap L_1=\{v\}$ \+\+\\

delete $v$ from $G$ and delete all triples containing $v$ from $K$;\-\\
 
2. While there exists $v\in L_1$ such that $v$ is in two or more distinct triples of $K$\+ \\

delete $v$ from $G$ and delete all triples containing $v$ from $K$;\-\\

3. While $\{u, v, w\} \in K$\+\\

delete $v \in L_1$ from $G$ and delete $\{u, v, w\}$ from $K$.\\
\end{tabbing}

Note that if $v\in L_1$ is deleted from $G$, all $L_2$ vertices in a triple with $v$ are exposed. Therefore, the undeleted vertices induce an outerplanar subgraph of $G$.

In the first two steps, at least two $L_2$ vertices were exposed for every deleted $L_1$ vertex. In the final step, all triples are disjoint, so each deletion of an $L_1$ vertex exposes one $L_2$ vertex and ensures that one $L_1$ vertex will not be deleted; again, 2 vertices are included in the induced outerplanar subgraph for every deleted vertex. This means that the subgraph contains at least two thirds of the vertices of $G$. This complete the proof of Theorem~\ref{thm:2outer}. \qed

\bigskip
This result is tight, as the disjoint union of multiple octahedrons (see Figure~\ref{fig:for}) is 2-outerplanar, and its largest induced outerplanar subgraph is on $\frac{2}{3}$ of its vertices. The result is also tight for arbitrarily large  connected 2-outerplanar graphs, as the same property holds for graphs constructed by connecting disjoint octahedrons as shown in Figure~\ref{fig:manyoct}.

\begin{figure}
\centering
\includegraphics[height=1.5in]{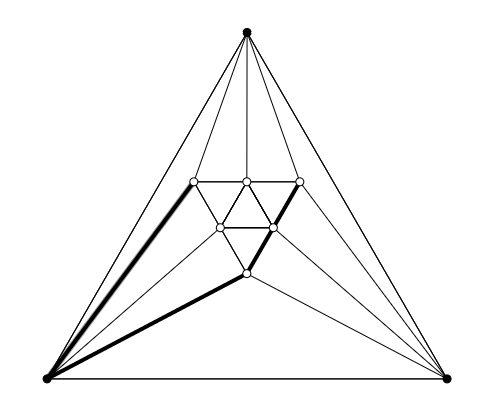}
\caption{The outerplanar induced subgraph of this graph found by the algorithm in Theorem~\ref{thm:2outer} is induced by the white vertices. Every induced forest on at least half of the vertices of this graph (an example is shown by the bolded edges) includes vertices not in this outerplanar subgraph.}
\label{fig:for}
\end{figure}

\begin{figure}
\centering
\includegraphics[height=1.5in]{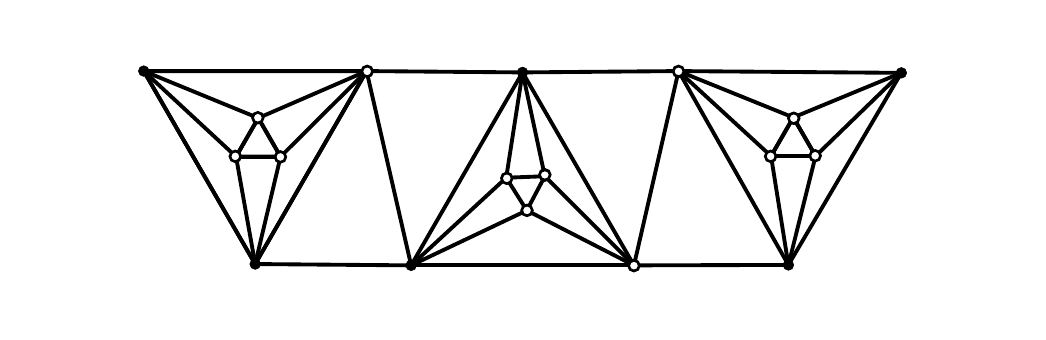}
\caption{A 2-outerplanar graph whose largest induced outerplanar subgraph is on $\frac{2}{3}$ of its vertices. The white vertices induce such a subgraph, found by the algorithm in Theorem~\ref{thm:2outer}.}
\label{fig:manyoct}
\end{figure}

%
%

\section{Future directions}

We define an \emph{induced outerplane graph} of a planar graph $G$ is an induced subgraph of $G$ whose embedding inherited from $G$ is an outerplanar embedding. We point out that our last result implies an improvement to a graph drawing result of  Angelini, Evans, Frati, and Gudmundsson~\cite{AEFG16} for the class of 2-outerplanar graphs. A simultaneous embedding with fixed edges and without mapping (\textsc{SEFENoMap}) of two planar graphs $G_1$ and $G_2$ of the same size $n$ is a pair of planar drawings of $G_1$ and $G_2$ that maps any vertex of $G_1$ into any vertex of $G_2$ such that: (i) vertices of both graphs are mapped to the same point set in a plane and (ii) every edge that belongs to both $G_1$ and $G_2$ must be represented by the same curve in the drawing of two graphs. The \textsc{OptSEFENoMap} problem asks for the maximum $k \leq n$ such that: given any two planar graphs $G_1$ and $G_2$ of size $n$ and $k$, respectively, there exists an induced subgraph $G_1'$ of $G_1$ such that $G'_1$ and $G_2$ have a \textsc{SEFENoMap} where the drawing of $G'_1$ is inherited from a planar drawing of $G$. The result of Gritzmann et al.~\cite{GMPP91}, implies that $k$ can be as large as the size of any induced outerplane graph of $G_1$. Angelini, Evans, Frati, and Gudmundsson (Theorem 1~\cite{AEFG16}) showed that any planar graph $G$ of size $n$ has an induced outerplane graph of size at least $\lceil n/2 \rceil$ which implies $k \geq \lceil n/2 \rceil$ by the result of Gritzmann et al.~\cite{GMPP91}. Our Theorem~\ref{thm:2outer} implies the following corollary, which is an improvement of the result of Angelini, Evans, Frati, and Gudmundsson for the class of 2-outerplanar graphs. 

\begin{corollary}
Every $n$-vertex 2-outerplanar graph and every $\lceil 2n/3 \rceil$-vertex planar graph have a \textsc{SEFENoMap}.
\end{corollary}

\noindent Based on Theorem~\ref{thm:2outer}, we conjecture that:
\begin{conjecture}\label{conj:3-outer}
Any 3-outerplanar graph on $n$ vertices contains an induced outerplane graph of size at least $\frac{2n}{3}$.
\end{conjecture}

If this conjecture is true, it would, by Hosono's result~\cite{Hosono}, imply that the largest induced forest of $3$-outerplanar graphs on $n$ vertices has size at least $\frac{4n}{9}$, that is an improvement over Borodin's result. It also improves the result of Angelini, Evans, Frati, and Gudmundsson for 3-outerplanar graphs. We note that in the proof of Theorem~\ref{thm:2outer}, we only need to delete vertices in $L_1$, and leave $L_2$ untouched, to get a large induced outerplane graph of 2-outerplanar graphs. For $3$-outerplanar graph, one may need to delete vertices in $L_3$ as shown by Figure~\ref{fig:3-outer-ex}.
 
\begin{figure}
\centering
\includegraphics[height=1.5in]{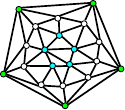}
\caption{A 3-outerplanar graph with $20$ vertices. Filled vertices are in odd layers and hollow vertices are in even layers. To obtain an induced outerplane graph of size at least $14$ ($\lceil 2n/3 \rceil$) vertices, one needs to delete at least one vertex in the innermost layer}.
\label{fig:3-outer-ex}
\end{figure}

We also believe that following conjecture, which is also mentioned in in~\cite{AEFG16}, is true:

\begin{conjecture}\label{conj:planar-outer}
A planar graph on $n$ vertices contains an induced outerplane graph of size at least $\frac{2n}{3}$.
\end{conjecture}

If this conjecture is true, it would imply an improvement of Borodin's result and Angelini, Evans, Frati, and Gudmundsson' result for general planar graphs. 

\paragraph{Acknowledgments:} The authors thank conversations with Bigong Zheng and Kai Lei. We thank anonymous reviewers for comments that help improving the presentation of this paper. 

 \bibliographystyle{plain}
\bibliography{bib}

\end{document}